\documentclass[preprint]{imsart}

\usepackage[numbers]{natbib}
\usepackage[OT1]{fontenc}
\usepackage{amssymb, amsmath, amsfonts,amsthm}
\usepackage[colorlinks]{hyperref}
\usepackage{graphicx}

\usepackage{amsfonts}
\usepackage{epstopdf}
\usepackage{subfigure}
\usepackage{amssymb}
\usepackage{amsmath}
\usepackage{url}
\usepackage{mathabx}
\usepackage{prodint}

\arxiv{math.PR/0000000}

\startlocaldefs
\numberwithin{equation}{section}
\theoremstyle{plain}
\newtheorem{theorem}{Theorem}[section]
\newtheorem{corollary}{Corollary}[section]
\newtheorem{lemma}{Lemma}[section]
\newtheorem{lemmaa}{Lemma A.\hspace{-1ex}}
\newtheorem{remark}{Remark}[section]
\endlocaldefs

\begin{document}
\def\N{\mathbb{N}}
\def\di{\displaystyle}
\def\indi{1\hspace{-1,1mm}{\rm I}}
\def\Fb{{\mathbb{F}}}
\def\Rb{{\mathbb{R}}}
\def\Pr{{\mathbb{P}}}
\def\thetab{\mbox{\boldmath$\theta$}}
\def\thetas{\mbox{\scriptsize\boldmath$\theta$}}
\def\zetab{\mbox{\boldmath$\zeta$}}
\def\Pc{{\cal{P}}}
\def\E{{\mathbb E}}
\def\Var{{\mathbb V}\hspace{-0.2ex}\mbox{\bf ar}}

\begin{frontmatter}
\title{{\large\bf UNIFORM CONVERGENCE RATE OF NONPARAMETRIC MAXIMUM LIKELIHOOD ESTIMATOR FOR THE 
CURRENT STATUS DATA WITH COMPETING RISKS\thanksref{T1}}}
\runtitle{Current status data with competing risks}
\thankstext{T1}{This research was supported in part by St.-Petersburg State University (grant No 1.52.1647.2016). }

\begin{aug}
\author{\fnms{Sergey V.} \snm{Malov}\thanksref{t1}\ead[label=e1]{malovs@sm14820.spb.edu}}
\thankstext{t1}{E-mail: malovs@sm14820.spb.edu  }

\runauthor{S.V. Malov}

\affiliation{St.-Petersburg State University and St.-Petersburg Electrotechnical University}

\address{
 St.-Petersburg State University $\&$  St.-Petersburg  Electrotechnical  University, 
 St.-Petersburg, Russia
 }

\end{aug}

\begin{abstract}
We study the uniform convergence rate of the nonparametric maximum likelihood estimator (MLE) for the sub-distribution functions in the current status data with competing risks model. It is known that the MLE have $L^2$-norm convergence rate $O_P(n^{-1/3})$ in the absolutely continuous case, but there is no arguments for the same rate of uniform convergence. We specify conditions for the uniform convergence rate $O_P(n^{-1/3}\log^{1/3} n)$ of the MLE for the sub-distribution functions of competing risks on finite intervals. The obtained result refines known uniform convergence rate in the particular case of current status data. The main result is applied in order to get the uniform convergence rate of the MLE for the survival function of failure time in the current status right-censored data  model. 
\end{abstract}


\end{frontmatter}

\section{Introduction}
\label{sec:intro}

We study a current status data with $K$ competing risks. The competing risks data is given as a sample from the bivariate distribution $(X,Y)$, where $X$ is a failure time variable and $Y\in\{1,\ldots,K\}$ is the corresponding failure cause. We assume that the failure time is not observed exactly, but at some random inspection time $T$. The current status observation with competing risks is $(T,\Delta)$, where $\Delta=(\Delta^1,\ldots,\Delta^{K+1})$, $\Delta^k=\indi_{\{X\leq T,Y=k\}}$ for $k=1,\ldots,K$, and $\Delta^{K+1}=\indi_{\{X>T\}}$. The observed data  is a sample from the distribution $(T,\Delta)$. The primary targets of statistical analysis are the sub-distribution functions $F_1,\ldots,F_K$ of the competing risks, $F_k(x)=\Pr(X\leq x,\Delta^k=1)$, $k=1,\ldots,K$.

The current status right-censored data is the special case of the current status data with two competing risks. Let the failure time $T^{\circ}$ be subject to random censoring by a random variable $U^{\circ}$. The event time $X=T^{\circ}\wedge U^{\circ}$ is not observed exactly, but  in a random inspection time $T$.  If both failure and censoring times fall before the observation time a current status of participant at the event time can assumed to be observed or not observed. The most interesting case of observed current status after censoring we call the current status interval right-censored data. The current status interval right-censored observation is  $(T,\Delta)$, where $\Delta^{1}=\indi_{\{T^{\circ}\leq U^{\circ}\leq T\}}$, $\Delta^{2}=\indi_{\{U^{\circ}<T^{\circ}\leq T\}}$ and $\Delta^{3}=1-\Delta^{1}-\Delta^{2}=\indi_{\{X>T\}}$.
Unlike the current status data with competing risks model, the primary target of interest now is the survival function $S$ of failure time $T^{\circ}$. 

The  current status data \cite{tur76} is the particular case of the current status data with competing risks under $K=1$, as well as the particular case of the current status right-censored data under $U^{\circ}=\infty$. The  nonparametric maximum likelihood estimator (MLE) for the current status data can be obtained as a solution of the isotonic regression model \cite{abers55} using Convex Minorant Algorithm. Alternatively, the MLE can be obtained by the EM-algorithm \cite{tur74,tur76}. Asymptotic behavior of the MLE at any fixed point studied in \cite{gre87,gre91,GW92}.  \citet{GW92} (see also \cite{gre91}) discussed wide range of asymptotic results on the MLE. Particularly, the uniform rate of convergence for the MLE of the failure time distribution function is obtained in \citet[Lemma~5.9]{GW92}. 

The MLE and the nonparametric pseudo likelihood estimator (PLE) of parameters for the current status data with competing risks, and the EM-algorithms to get the estimators are given by \citet{HSL01}. Another na\"ive (ad-hoc) estimator is discussed in \cite{JLH03}, as well as the MLE. Consistency and rate of convergence results for the MLE are obtained in \cite{GMW08a}, and weak convergence results are given in  \cite{GMW08b}. The current status data model with two competing risks is a baseline statistical model for the current status interval right-censored data model. Then the distribution of failure time is restored from the baseline  parameter using the product-limit method. Consistency and the rate of convergence in total variance of the corresponding MLE, PLE and the na\"ive estimators for the survival function of failure time in the current status right-censored data model are obtained in \cite{mal19}. 

In this work we focus on the uniform rate of convergence of the MLE for the distributions of competing risks. The obtained uniform convergence rate will be used to improve the rate of convergence result in \cite{mal19}.  The MLE of the competing risks distributions in the current status data with competing risks model and the corresponding estimate of the survival function of failure time for the current status right-censored data are described in Section \ref{sec:mlik}. In Section \ref{sec:rate} we discuss the uniform convergence rate of the MLE for competing risks distribution functions and the corresponding survival function of failure time in the current status right-censored data model.  Main proofs are given the last Section \ref{sec:proof}, and a technical lemma is postponed to Appendix. 

\section{The maximum likelihood estimate}
\label{sec:mlik}

In this section we study the likelihood function in the current status data with competing risks model and discuss the MLE for the distribution functions of the competing risks as well as the MLE for survival functions of failure time in the current status right-censored data model. 

Assume that the competing risk $(X,Y)$ is independent of the observation time $T$.
Let $(T_i,\Delta_i)$,  where $\Delta_i=(\Delta_i^{1},\ldots,\Delta_i^{K})$, $i=1,\ldots,n$, be a sample from the distribution $(T,\Delta)$; $\Fb$ be the set of $K$-tuples $F=(F_1,\ldots,F_K)$ of sub-distribution functions (non-negative nondecreasing grounded at $0$ cadlag) with $F_+\equiv \sum_{k=1}^K F_{i}\leq 1$;
$(F_{01},\ldots,F_{0K})\in\Fb$ be the true sub-distribution functions of the competing risks;
$
\gamma_{k}\!=\!\sup\{x\!:\! F_{0k}(x)\!<\!F_{0k}(\infty)\}
$;
$F_{0,K+1}\equiv 1-F_{0+}$ and $F_{0+}\equiv \sum_{k=1}^K F_{0k}$. The log-likelihood function for the current status data  with competing risks is following:
\setlength{\arraycolsep}{2pt}
\begin{equation}
\label{equ:llik0}
LL_n(F) =\displaystyle \int_{\Rb\times \{0,1\}^K}\Bigl\{\sum\nolimits_{k=1}^{K} \delta_k\log F_k(t)+\bar\delta \log F_{K+1}(t)\Bigr\}dP_n(t,\delta),
\end{equation}
where $\bar\delta=1-\sum_{k=1}^K \delta_k$, and $P_n$ is the empirical measure of the sample $(T_i,\Delta_i)$, $i=1,\ldots,n$. Let $T_{(1)}\leq\ldots\leq T_{(n)}$ be the order statistics of the sample $T_1,\ldots,T_n$ and $\Delta_{(1)},\ldots,\Delta_{(n)}$, where $\Delta_{(i)}=(\Delta_{(i)}^1,\ldots,\Delta_{(i)}^K)$, $i=1,\ldots,n$, be the corresponding concomitants. Denote $\Fb_n$ is the set of $K$-tuples of sub-distribution step functions $(F_{n1},\ldots,F_{nK})$: $F_{nk}$ has jumps on the set of observation times $\{T_{(i)}:\Delta_{(i)}^k=1\}$ and $\sum_{k=1}^K F_{nk}(\infty)\leq 1$. The MLE $\widehat F_n=(\widehat F_{n1},\ldots,\widehat F_{nK})$ maximizes the log-likelihood (\ref{equ:llik0}) over $\Fb_n$. 

The characterization of the MLE $\widehat F_n$ due to \citet[Corollary 2.10]{GMW08a}. The $\widehat F_{n}$ is maximizes $LL_n$ over the set of functions $F_n\in\Fb_n$ iff for $k=1,\ldots,K$ at each jump-point $\tau_{nk}$ of $\widehat F_{nk}$
\begin{equation}
\label{equ:char}
\int_{[\tau_{nk},s)} \Bigl\{\frac{dV_{nk}(u)}{\widehat F_{nk}(u)} - \frac{dV_{n,K+1}(u)}{\widehat F_{n,K+1}(u)}\Bigr\}\geq
\beta_n \indi_{[\tau_{nk},s)}(T_{(n)}),\quad s\in\Rb,
\end{equation}
with the equality holds if $s$ is a point of increase of $\widehat F_{nk}$ and $s>T_{(n)}$,  
where $V_{nk}(u)=\int_{t\leq u}\delta_k dP_n(t,\delta)$, $k=1,\ldots,K+1$, and $\beta_n=1-\int \frac{dV_{n,K+1}(u)}{\widehat F_{n,K+1}(u)}$. Moreover, $\beta_n\geq 0$, and $\beta_n=0$ iff there exists an observation $T_i=T_{(n)}$, such that $\Delta_{i}^{K+1}=1$ \cite[Corollary 2.9]{GMW08a}. The inequality (\ref{equ:char}) implies immediately that at each jump-point $\tau_{nk}$ of $\widehat F_{nk}$
\begin{equation}
\label{equ:char1}
\int_{[\tau_{nk},s)} \Bigl\{\frac{dV_{nk}(u)}{\widehat F_{nk}(u)} - \frac{dV_{n,K+1}(u)}{\widehat F_{n,K+1}(u)}\Bigr\}\geq
0,\quad s<T_{(n)},
\end{equation}
and at each jump-point $\tau_{nk}<T_{(n)}$ of $\widehat F_{nk}$
\begin{equation}
\label{equ:char2}
\int_{[s,\tau_{nk})} \Bigl\{\frac{dV_{nk}(u)}{\widehat F_{nk}(u)} - \frac{dV_{n,K+1}(u)}{\widehat F_{n,K+1}(u)}\Bigr\}\leq
0,\quad s\geq T_{(1)},
\end{equation}
with the equalities hold if $s$ is a point of increase of $\widehat F_{nk}$, $k=1,\ldots,K$.  

In order to recover the survival function $S$ of failure time $T^{\circ}$ in the current status right-censored data model which is based on the current status data with two competing risks, one can use the following representation of the cumulative hazard function $\Lambda(x)=\int_0^x (1-F_{01-} -F_{02-})^{-1} dF_{01}$ and, therefore, 
\begin{equation}
\label{equ:reconstr}
S(t)=\PRODI_{x\leq t}\Bigl(1-\frac{dF_{01}(x)}{1-F_{01}(x_-)-F_{02}(x_-)} \Bigr) 
\end{equation}
under $F_{01}(x)=\Pr(T^{\circ}\leq x, T^{\circ}\leq U^{\circ}\leq T)$ and $F_{02}(x)=\Pr(U^{\circ}\leq x, U^{\circ}<T^{\circ}\leq T)$. The survival function $Q$ of censoring time $U^{\circ}$ is determined by the cumulative hazard function
$\Lambda^{\! U}(x)=\int_0^x \frac{S_-}{S F_{03-}} \,dF_{02}$ and, therefore, 
$
Q(t)=\prodi\nolimits_{x\leq t}\bigl(1-d\Lambda^{\! U}(x) \bigr).
$ 
Alternatively, 
$
Q(t)=\int_0^t 1/S\, dF_{02}.
$

There are several ways to get the MLE from current status data with competing risks. The EM-algorithm due to \citet{HSL01}  is working too slow. It would be preferable to use the iterated convex minorant (ICM) algorithm (see  \citet{GJ14}, Section 7.5) based on the characterization  of the MLE for current status data with competing risk in (\ref{equ:char}). Alternatively, the MLE for the parameter $F$ can be obtained by applying the support reduction algorithm \cite{GJW08} realized in the R-package {\it MLEcens} \cite{mth13}. In order to create the MLE $\widehat S_n$ for the survival function of failure time $T^{\circ}$ in the current status right-censored data model one can apply the reconstruction formula (\ref{equ:reconstr}) with $(\widehat F_{n1}, \widehat F_{n2})$ instead of $(F_{01}, F_{02})$.  

\section{The uniform convergence rate}
\label{sec:rate}

In this section we discuss the uniform convergence rate of the MLE for the current status data with competing risks. Moreover, we obtain the uniform convergence rate for the survival function of failure time in the current status right-censored data model as an application of the result for current status data with competing risks. We will slightly abuse notation by using the same symbol for a non-decreasing function and the induced Lebesgue--Stieltjes measure. Particularly, $F_k((-\infty,x])=F_k(x)$ for all $x\in\Rb$, $k=1,\ldots,K+1$. 

For each $F\in \Fb$ we define $L_{F}\!\!:\!\Rb\times\{0,\!1\}^K\to \Rb_+$ as
$
L_{F}=L_{F}(w,\delta)=\prod_{i=1}^{K+1} F_k(w)^{\delta_k}
$,
and $\Pc=\{L_{F}:F\in \Fb\}$. Introduce the Hellinger  distance between two functions $p_1\in\Pc$ and $p_2\in\Pc$ as
$$
h(p_1,p_2)=\Bigl(\frac{1}{2}\int (p_1^{1/2}-p_2^{1/2})^2d\mu\Bigr)^{1/2},
$$
where $\mu=G\times\nu_1\times\ldots\times\nu_K$, $G$ is the distribution of $T$, and $\nu_1,\ldots,\nu_K$ are the counting measures on $\{0,1\}$. 
We also use notations 
$\|\cdot\|_{2}=\bigl(\int \|\cdot\|^2 dG\bigr)^{1/2}$ is the $L_2(G)$-norm, $\|\cdot \|_{A}=\sup_{A}\|\cdot\|$ and $\|\cdot\|=\|\cdot\|_{\Rb}$ for the supremum norm. 

\citet[Theorem 4.1]{GMW08a} obtained Hellinger rate of convergence $h(L_n,L_0)=O_P(n^{-1/3})$ that implies immediately 
\begin{equation}
\label{equ:qrate}
\|\widehat F_{nk}-F_{0k}\|_{2}=O_P(n^{-1/3}), 
\end{equation}
but there is no arguments for the same rate of uniform convergence. In the particular case of interval censored data the uniform convergence rate $O(n^{-1/3}\log n)$ obtained by \citet[Lemma 5.9]{GW92}. \citet[Theorem 4.10]{GMW08a} show that under continuously differentiable $F_{0k}$ and $G$ with bounded away from zero derivatives at some fixed point $t_0$, there exists a constant $r>0$ such that 
$$
\sup_{t\in [t_0-r,t_0+r]}\frac{|\widehat F_{n+}(t)-F_{0+}(t)|}{v_n(t-t_0)}=O_P(1),
$$
where $v_n(t)=n^{1/3}\indi_{\{t\leq n^{-1/3}\}}+n^{(1-\beta)/3}|t|^{\beta}\indi_{\{t> n^{-1/3}\}}$ for $t>0$ and some $\beta\in (0,1)$. The uniform convergence rate $O_P(n^{-(1-\beta)/3})$ of the MLE $\widehat F_n$ to the parameter $F_0$ on any interval $[\gamma_-,\gamma_+]$, such that $F_{0k}\in (0,1)$, $k=1,\ldots,N$, and $G\in (0,1)$ both are continuously differentiable with bounded away from zero derivatives on the interval, then follows immediately, but it does not imply the uniform convergence in a neighborhood of point~$0$. The main result of this work is following.
 
\begin{theorem}
\label{teo:unirate}
Let $F_{0+}\equiv\sum_{k=1}^K F_{0k}$; $\gamma: F_{0+}(\gamma)<F_{0+}(\infty)$; the functions $F_{0k}$ and $G$ are absolutely continuous, $F_{0k}<\!\!<G$ with $\varepsilon\leq \frac{dF_{0k}}{dG}\leq 1/\varepsilon$ on the interval $(0,\gamma]$ for some $\varepsilon\in (0,1)$, $k=1,\ldots,K$. Then for all $k=1,\ldots,K$,
\begin{equation*}
\|\widehat F_{nk}-F_{0k}\|_{[0,\gamma]}=O_P(n^{-1/3}\log^{1/3} n).
\end{equation*}
\end{theorem}

In the particular case of interval censored data ($K=1$) we use the notations $F_0$ is the true distribution function of failure time and $\widehat F_n$ is the corresponding MLE. The refined uniform rate of convergence result for the MLE in the interval censored data model is given in the following corollary.

\begin{corollary}
\label{cor:unirate}
Let $K=1$, and the conditions of Theorem \ref{teo:unirate} hold uniformly for all $\gamma<\gamma_+$. Then
\begin{equation*}
\|\widehat F_n-F_0\|=O_P(n^{-1/3}\log^{1/3} n). 
\end{equation*}
\end{corollary}

\noindent
\begin{remark} 
The uniform convergence rate in Corollary \ref{cor:unirate}  is more precise then one obtained in \citet[Lemma 5.9]{GW92}. 
\end{remark}

The $L^1(G([0,\gamma]))$-norm rate of convergence result for the MLE of the survival function $S$ in the current status right-censored data model obtained by \citet{mal19}. Here we apply Theorem \ref{teo:unirate} in order to get the same rate of uniform convergence on the interval $[0,\gamma]$. 

\begin{corollary}
\label{teo:funirate}
Let \ $G$ \ is absolutely \ continuous; \ $S^*<\!\!< G$, \ $Q^*<\!\!< G$ \ and $\varepsilon\leq\frac{dS^*}{dG},\frac{dQ^*}{dG}\leq 1/\varepsilon$ on the interval $[0,\gamma]$ for some $\varepsilon>0$ and $\gamma<\gamma_+$, where $S^*\equiv 1-S$ and $Q^*\equiv 1-Q$. Then
\begin{equation}
\label{equ:rate}
\|\widehat S_n-S\|_{[0,\gamma]}=O_P(n^{-1/3}\log^{1/3} n).
\end{equation}
\end{corollary}

\section{Proofs}
\label{sec:proof}

In order to prove Theorem \ref{teo:unirate} we need several auxiliary results. The following local convergence result is quite different to \citet[Theorem 4.10]{GMW08a}, but its proof is very similar. 

\begin{lemma}
\label{lem:lunirate}
Let $0\leq F_{0+}(t_0)<F_{0+}(\infty)$; $G$ and $F_{0k}$, $k=1,\ldots,K$, be continuously differentiable at $t_0$ with positive and bounded away from zero derivatives in a neighborhood $V_{r}(t_0)$ for some $r>0$, where $V_{r}(t_0)=(t_0-r,t_0+r)$ if $F_{0+}(t_0)>0$, and $V_{r}(t_0)=(t_0,t_0+r)$ if $F_{0+}(t_0)=0$. Then there exists a constant $r>0$, such that 
\begin{equation}
\label{equ:unilocrate}
\sup\nolimits_{t\in V_{r}(t_0)} |\widehat F_{n+}(t)-F_{0+}(t)|=O_P(n^{-1/3}\log^{1/3} n),
\end{equation}
where $\widehat F_{n+}\equiv\sum_{k=1}^K \widehat F_{nk}$.
\end{lemma}

Let $\tau_{nk1}<\ldots<\tau_{nkm_k}$ be the successive jump points  of $\widehat F_{nk}$. Taking account of $\widehat F_{nk}(t)=\widehat F_{nk}(\tau_{nki})$ for all $t\in [\tau_{nki},\tau_{nk,i+1})$ we get from (\ref{equ:char1}) that for all $s<T_{(n)}$, for any point of jump $\tau_{nk}$ of $\widehat F_{nk}$
\begin{equation}
\label{equ:ch1}
\int_{[\tau_{nk},s)}\delta_i\, dP_n(t,\delta)-\int_{[\tau_{nk},s)}\frac{\widehat F_{nk}(t)\bar\delta}{\widehat F_{n,K+1}(t)}dP_n(t,\delta)\geq 0
\end{equation}
with the equality holds if $s$ is a point of jump of $\widehat F_{nk}$. The inequality (\ref{equ:ch1}) is applicable to obtain the local uniform rate of convergence result for any point $t_0<\gamma$, unlike (\ref{equ:char1}), which is not applicable under $F_{0+}(t_0)=0$. 

In order to prove Lemma \ref{lem:lunirate} we are following  \citet[proof of Theorem 4.10]{GMW08a} with another rate of convergence $a_n=n^{-1/3}\log^{1/3} n$ (instead of $v_n(t)$ in \citet{GMW08a}, equation (31)) and another martingales 
\begin{equation}
\label{equ:mart1}
\begin{array}{rcl}
M_{nk}(t)&=&\di\int_{u\leq t} \!(\delta_k-F_{0k}(u))dP_n(u,\delta)\\
&-&\di\int_{u\leq t}\!\frac{F_{0k}(u)(\bar\delta-F_{0,K+1}(u))}{F_{0,K+1}(u)}dP_n(u,\delta)
\end{array}
\end{equation}
(instead of \citet{GMW08a}, equation (18)). The following lemmas precedes the proof of Lemma \ref{lem:lunirate}. 

\begin{lemma}
\label{lem:martingale}
Let $\gamma<\gamma_+$ be a fixed constant. Then under the conditions of Lemma \ref{lem:lunirate}, at each jump point $\tau_{nk}$ of $\widehat F_{nk}$ 
\begin{equation}
\label{equ:mart-a}
\begin{array}{r}
\di 
\int_{[\tau_{nk},s)} \Bigl((\widehat F_{nk}(t)-F_{0k}(t)) + \frac{F_{0k}(t)(\widehat F_{n+}(t)-F_{0+}(t))}{F_{0,K+1}(t)}\Bigr)dG(t)\\
\di\leq \int_{[\tau_{nk},s)}  dM_{nk}(t)+{\cal R}_{nk}(\tau_{nk},s)
\end{array}
\end{equation}
and \vspace{-3mm}
\begin{equation}
\label{equ:mart-b}
\begin{array}{r}
\di 
\int_{[t,\tau_{nk})} \Bigl((\widehat F_{nk}(t)-F_{0k}(t)) + \frac{F_{0k}(t)(\widehat F_{n+}(t)-F_{0+}(t))}{F_{0,K+1}(t)}\Bigr)dG(t)\\
\di\geq \int_{[t,\tau_{nk})}  dM_{nk}(w)+{\cal R}_{nk}(t,\tau_{nk})
\end{array}
\end{equation}
for all $s<T_{(n)}$ and $t>T_{(1)}$, where
\begin{equation*}
\sup\nolimits_{t,s\in V_{r}(t_0): t<s} (|{\cal R}_{nk}(t,s)|)= O_P(n^{-2/3})
\end{equation*} 
for all $k=1,\ldots,K$ and some $r>0$. 
\end{lemma}

\begin{proof}
Taking into account (\ref{equ:mart1}) the left hand side of (\ref{equ:ch1}) can be rewritten as $M_{nk}([\tau_{n,i},s))-I_n(\tau_{n,i},s)$, where
$$
I_n(t,s)=\int_{[t,s)} \Bigl(\frac{\widehat F_{nk}(u)}{\widehat F_{n,K+1}(u)}- \frac{F_{0k}(u)}{F_{0,K+1}(u)}\Bigr)\bar\delta \,dP_n(u,\delta), 
$$
and $I_n(t,s)=I_n^{(1)}(t,s)+I_n^{(2)}(t,s)$, where 
$$
I_n^{(1)}(t,s)=\int_{[t,s)}\frac{\widehat F_{nk}(u)-F_{0k}(u)}{\widehat F_{n,K+1}(u)}\bar\delta\,dP_n(u,\delta),
$$
and
\vspace{-.3em}
$$
I_n^{(2)}(t,s)=\int_{[t,s)}\frac{F_{0k}(u)(\widehat F_{n+}(u)-F_{0+}(u))}{F_{0,K+1}(u)\widehat F_{n,K+1}(u)}\bar\delta \,dP_n(u,\delta). 
$$
Moreover,
$
I_n^{(1)}(t,s)=\int_{[t,s)}(\widehat F_{nk}(u)-F_{0k}(u))dG(u)+\rho^{(1)}(t,s)+\rho^{(2)}(t,s),
$
where \vspace{-.2em}
$$
\rho^{(1)}(t,s)=-\int_{[t,s)} \frac{(\widehat F_{nk}(u)-F_{0k}(u))(\widehat F_{n+}(u)-F_{0+}(u))}{
\widehat F_{n,K+1}(u) F_{0,K+1}(u)}\bar\delta \, dP_n(u,\delta),
$$
$$
\rho^{(2)}(t,s)=\int_{[t,s)} \frac{(\widehat F_{nk}(u)-F_{0k}(u))(\bar\delta dP_n(u,\delta)-F_{0,K+1}(u)dG(u))}{F_{0,K+1}(u)},
$$
and $I_n^{(2)}(t,s)=\int_{[t,s)}\frac{F_{0k}(u)(\widehat F_{n+}(u)-F_{0+}(u))}{F_{n,K+1}(u)}dG(u)+\rho^{(3)}(t,s)+\rho^{(4)}(t,s)$, where 
$$
\rho^{(3)}(t,s)=\int_{[t,s)} \frac{F_{0k}(u)(\widehat F_{nk}(u)-F_{0k}(u))^2}{F_{0,K+1}(u)^2\widehat F_{n,K+1}(u)}\bar\delta dP_n(u,\delta),
$$
$$
\rho^{(4)}(t,s)\!=\!\int_{[t,s)}\!\!\! \frac{F_{0k}(u)(\widehat F_{n+}(u))\!-\! F_{0+}(u))(\bar\delta\, dP_n(u,\delta) \!-\! F_{0,K+1}(u)dG(u))}{F_{0,K+1}(u)^2}.
$$
Hence, ${\cal R}_{nk}(t,s)=\rho^{(1)}(t,s)+\rho^{(2)}(t,s)+\rho^{(3)}(t,s)+\rho^{(4)}(t,s)$. 

Note that $\E \int_0^t \bar\delta dP_n(t,\delta)=\int_0^t F_{0,K+1}(u)dG(u)$ for all $t\geq 0$. Taking account of $F_{0,K+1}(s)>F_{0,K+1}(\gamma)=\epsilon_{\gamma}$ for some $\epsilon_{\gamma}>0$ and consistency of $\widehat F_{n+}$ [\citet[Proposition 3.3]{GMW08a}] we can write that 
$$
\begin{array}{l}
\di
|\rho^{(1)}(s,t)| 
\leq 4\epsilon_{\gamma}^{-2}\Bigl| \int_{[t,s)} (\widehat F_{nk}(u)-F_{0k}(u))(\widehat F_{0+}(u)\\
\hspace{+30.8ex}
\di -\,F_{0+}(u))(\bar\delta\, dP_n(u,\delta) \!-\! F_{0,K+1}(u)dG(u))\Bigr|
\\\di \hspace{+10.8ex}
+\, 2\epsilon_{\gamma}^{-1}\int_{[t,s)} |\widehat F_{nk}(u) -F_{0k}(u)||\widehat F_{n+}(u)-F_{0+}(u)|dG(u)
\end{array}
$$ 
for sufficiently large $n$ almost sure. Then we apply \cite[Lemma 5.13]{vdg00} with $\alpha=1$ and $\beta=0$ and (\ref{equ:qrate}) to obtain the required rate of convergence $O_P(n^{-2/3})$ for the first summand in the right hand side of the last inequality. By the Cauchy-Schwarz inequality and (\ref{equ:qrate}),
\begin{eqnarray*}
\int_{[t,s)}\!\! \!\!|\widehat F_{nk}(u)\!-\! F_{0k}(u)||\widehat  F_{n+}(u)&-&F_{0+}(u)|dG(u)\\
&\leq &{\|\widehat  F_{nk}\!- F_{0k}\|}_2 {\|\widehat  F_{n+}\!- F_{0+}\|}_2\!=\!O_P(n^{-2/3}).
\end{eqnarray*}
Therefore, $\rho^{(1)}(t,s)=O_P(n^{-2/3})$ uniformly for all $t,s\in V_{r}(t_0)$: $t<s$. Similarly, we obtain  $\rho^{(3)}(t,s)=O_P(n^{-2/3})$, and $\rho^{(2)}(t,s)=O_P(n^{-2/3})$, $\rho^{(4)}(t,s)=O_P(n^{-2/3})$ follows immediately from (\ref{equ:qrate}) by \citet[Lemma 5.13]{vdg00}. Hence, $\sup_{t,s\in V_{r}(t_0): t<s}|{\cal R}_{nk}(t,s)|=O_P(n^{-2/3})$ for some $r>0$. Finally, (\ref{equ:char1}) implies (\ref{equ:mart-a}), and (\ref{equ:char2}) implies (\ref{equ:mart-b}).
The lemma is proved. 
\end{proof}

\begin{lemma}
\label{lem:martbase}
Under the conditions of Lemma \ref{lem:lunirate} for any $b>0$, $s_n\in V_{r}(t_0)$
\begin{equation}
\label{equ:marti1}
\Pr\Bigl(\sup_{w\in V_{r}(t_0):w<s_n-Ma_n}\Bigl\{
\int_{[w,s_n)} dM_{nk}-b(s_{n}-w)^2\Bigr\}\geq 0\Bigr)\leq p_{jbM}
\end{equation}
and 
\begin{equation}
\label{equ:marti2}
\Pr\Bigl(\sup_{w\in V_{r}(t_0):w\geq s_n+Ma_n}\Bigl\{
\int_{(s_n,w]} dM_{nk}+b(w-s_{n})^2\Bigr\}\leq 0\Bigr)\leq p_{jbM},
\end{equation}
where $a_n\!=\!n^{-1/3}\log^{1/3}n$ and $p_{jbM}\!=\!d_1\exp(-d_{2b} M^3\log n)$ for some \hbox{$d_1,d_{2b}\!>\!0$.}
\end{lemma}

\begin{proof}
In order to prove (\ref{equ:marti1}) we set $t_{n0}=s_n-Ma_n$ and 
$
J_{nq}=[t_{nq},t_{n,q-1})
$,
where $t_{nq}=t_0-n^{-1/3} q$, $q=1,\ldots,q_{nr M}$, and $q_{nr M}:t_{n,j_{nr M}}\notin V_{r}(t_0)$. Then the left hand side of (\ref{equ:marti1}) is bounded above by 
\begin{equation}
\label{equ:martt1}
\sum\nolimits_{q=1}^{q_{nr M}} \Pr\Bigl(\sup_{t\in J_{nq}}\Bigl\{
\int_{[t,s_n)} dM_{nk}\geq b(s_{n}-t)^2\Bigr\}\Bigr).
\end{equation}
Introduce for each $\theta>0$ the reverse submartingale $\exp\bigl(n\theta \int_{[t,s_n)} dM_{nk}\bigr)$ for $t<s_n$ with respect to the filtration ${\cal F}_t=\{(T_i,\Delta_i),i=1,\ldots,n:T_i\geq t\}$. By Doob's submartingale inequality we obtain that 
\begin{eqnarray*}
&&\Pr\Bigl(\sup_{t\in J_{nq}}\Bigl\{
\exp\Bigl(nx\int_{[t,s_n)} dM_{nk}\Bigr)\geq \exp(nx b(s_{n}-t)^2)\Bigr\}\Bigr)\\
&&\hspace{+25ex}\leq  
\exp(-nx b(s_{n}-t)^2)\,\E \exp\Bigl(nx\int_{[t_{nq},s_n)} dM_{nk}\Bigr)
\end{eqnarray*}
Taking account of $S_{kn}$ is a sum of i.i.d. variables we can write that
$$
\E \exp\Bigl(n\theta\int_{[t_{nq},s_n)} dM_{nk}\Bigr)=\Bigl(\E \exp\bigl(\theta \indi_{[t_{nq},s_n)}(T) \zeta_{nk}(T,\delta)\bigr)\Bigr)^n,
$$
where \ \ $\zeta_{nk}(T,\Delta)=\Delta^k-\frac{F_{0k}(T)\Delta^{K+1}}{F_{0,K+1}(T)}$. \ \ Using \ \ the \ \ exponential \ \ series, $\E(\zeta_{nk}(T,\Delta)|T)=0$ and $\log(1+x)\leq x$ we obtain that the right hand side of the last equation is 
\begin{eqnarray*}
&&\exp\Bigl\{n\log\Bigl(1+\E\indi_{[t_{nj},s_n)}(T)\sum\nolimits_{l=2}^{\infty} \frac{x^l \zeta_{nk}(T,\Delta)^l}{l!}\Bigr)\Bigr\}\\
&&\hspace{+30ex}\leq 
\exp\Bigl\{\frac{1}{2}n x^2(s_n-t_{nq})f_n(x,t_{nq},s_n)\Bigr\},
\end{eqnarray*}
where 
$
f_n(x,c_1,c_2)=\frac{2}{(c_2-c_1)} \sum_{l=2}^{\infty} \frac{x^{l-2}}{l!}\int_{c_1}^{c_2} |\E( \zeta_{nk}(T,\Delta)|T=u)|dG(u).
$
Since $\zeta_{nk}(T,\Delta)$ given $T$ is bounded uniformly on $T\in V_{r}(t_0)$ and $x\mapsto f_n(x,c_1,c_2)$ is a continuous and strictly monotone increasing in $x$ function, there exists unique solution $
x_{c_1,c_2}$ of the equation  
$
x f_n(x,c_1,c_2)=b(c_2-c_1) 
$
and 
$
x_{c_1,c_2}\leq b(c_2-c_1)^2\Bigm/\int_{c_1}^{c_2} |\E( \zeta_{nk}(T,\Delta)^2|T=t)|dG(t).
$
Choosing $(c_1,c_2)=(t_{nq},s_n)$ we obtain that $q$-th summand in (\ref{equ:martt1}) is bounded above by 
$$
\exp\Bigl(-\frac{1}{2}nx_{t_{nq},s_n}b(s_n-t_{nq})^2\Bigr)\leq \exp\bigl(-nd_{2b}(s_n-t_{nq})^3\bigr),
$$where $d_{2b}=b^2 \bigm/\bigl(2\varepsilon\inf_{t\in V_{r}} F_{0,K+1}(t)\bigr)$. Taking account of $s_n-t_{nq}\geq Ma_n+q n^{-1/3}$ and $(Ma_n+qn^{-1/3})^3\geq (Ma_n)^3+q^3/n$ we obtain that 
$$
\exp\bigl(-nd_{2b}(s_n-t_{nq})^3\bigr)\leq \exp(-d_{2b} M^3\log n)\exp(-d_{2b}q^3).
$$
Hence, (\ref{equ:marti1}) holds with $d_1=\sum_{q=1}^{\infty} \exp(-d_{2b}q^3)$. 

The inequality (\ref{equ:marti2}) can be obtained analogously by using the similar grid $t_{nq}$ on the right of the point $t_0=s_n+Ma_n$ and applying Doob's submartingale inequality to the submartingale $\exp\bigl(-n\theta \int_{[s_n,t)} dM_{nk}\bigr)$ for $t\geq s_n$ with respect to the filtration ${\cal F}_t^*=\{(T_i,\Delta_i),i=1,\ldots,n:T_i\leq t\}$. The lemma is proved. 
\end{proof}

We continue with the proof of Lemma \ref{lem:lunirate}. 

\begin{proof}[Proof of Lemma \ref{lem:lunirate}] In order to prove the lemma we are actually mimics arguments used in \cite[proof of Theorem 4.10]{GMW08a}, so we just mention crucial points in our proof.  Let $a_n=n^{-1/3}\log^{1/3} n$. We focus on the case of $ F_{0+}(t_0)=0$, which is not covered in \cite{GMW08a}. 
In this case, 
$
F_{0+}(t+M a_n)\leq F_{0+}(t)+2M F'_{0+}(t_{0+}) a_n
$
and
$
F_{0+}(t-M a_n)\geq F_{0+}(t)-2M F'_{0+}(t_{0+}) a_n
$
for all $t\in (t_0,t_0+r)$ under sufficiently small $r$. Then it is sufficient to prove that for any $\epsilon>0$ there exist $n_0$ and $M>0$, such that 
\begin{equation}
\label{equ:scl1}
\Pr\bigl(\exists t\in (t_0,t_0+r):\widehat F_{n+}(t)\notin [F_{n+}(t-M a_n),F_{n+}(t+M a_n))\bigr)<\epsilon
\end{equation}
for all $n>n_0$. 

Note \ that \ the \ first \ jump \ point \ $\tau_{nk1}$ \ of \ $\widehat F_{nk}$ \ is \ the \ minimal \ $T_i$ \ with \ $\Delta_i^k=1$. Let
$m_{k}=\min\{i:\Delta_{(i)}^k=1\}$. By \citet{yng77},
\begin{eqnarray*}
&&\Pr(m_{k}>m) =\!\frac{n!}{(n\!-\! m)!}\\
&&\hspace{+4ex}\times\int_{\Rb^m}\Bigl\{\prod_{i=1}^{m} (1\!-\!F_{0k}(t_i))\Bigr\} (1\!-\!G(t_m))^{n-m}\indi_{\{t_1\leq\ldots\leq t_m\}}dG(t_1)\cdots dG(t_m).
\end{eqnarray*}
Assume for a moment that $F_{0k}\equiv G$ for all $k=1,\ldots,K$. Then
$$ 
\Pr(m_{k}> m) =  \frac{n!(n-m)!!}{(n-m)!(n+m)!!}\leq \Bigl(\frac{n}{n+m}\Bigr)^{[(m+1)/2]}. 
$$
Therefore, $\Pr(m_{k}>m)\to 0$ as $n\to\infty$, if $n^{1/2}/m\to 0$ as $n\to\infty$ and the density function $\frac{dF_{nk}}{dG}$ is bounded. Set $m=m(n)=n^{1/2}\log n$. Note that $\E(T^{u}_{(m)})=\frac{m}{n+1}=O(n^{-1/2}\log n)$ and $\Var(T^{u}_{(m)})=\frac{m(n-m+1)}{(n+1)^2(n+2)}=O(n^{-2}\log n)$ for the uniform order statistics $T^{u}_{(i)}=G(T_{(i)})$, $i=1,\ldots,n$. Then applying the Chebishev's inequality we obtain that for any fixed $c\!>\! 0$, $\Pr(T^{u}_{(m_k)}\!\!<\!c n^{-1/3})\!\to\! 1$ as $n\!\to\!\infty$. Taking account of $G'(t_{0+})\!>\!\varepsilon$ we conclude that $\Pr(T_{(m_k)}\!<\!t_0\!+\!n^{-1/3})\!=\!\Pr(\tau_{n,1}\!<\!t_{n,1})\!\to\! 1$ as $n\to\infty$.   

Now, applying Lemma \ref{lem:martingale} and \cite[Proposition 3.3]{GMW08a} we obtain that for any $\epsilon>0$ there exists $C>0$, such that $\Pr(B_{nr C})\geq 1-\epsilon/2$ for sufficiently large $n$ almost sure, where
$$
\begin{array}{l}
\displaystyle
B_{nrC}=\{\mbox{each of}\; F_{nk},k=1,\ldots,K,\; \mbox{has a jump in} (t_0+r,t_0+2r), 
\vspace{+0.3em}
\\
\displaystyle
\hspace{+2ex}
t_0\!+\! 2r\! <\! T_{(n)},
\max_{k=1,\ldots,K}\tau_{nk1}\!<\! t_{n,1},\!\!\!\sup_{0\leq t<s<t_0+2r} (\max_{k=1,\ldots,K}\! |{\cal R}_{nk}(t,s)|)\!\leq\!  C n^{-2/3}\}
\end{array}
$$
We split the interval $[t_0,t_0+r)$ to the subintervals $I_{n,j}=[t_{n,j},t_{n,j+1})$ for $j=0,\ldots,\lceil r n^{1/3}\rceil$, where $t_{n,j}=t_0+j n^{-1/3}$ and prove that 
\begin{equation}
\label{equ:prt}
P\bigl(\exists t\in I_{n,j}:\widehat F_{n+}(t)\notin [F_{0+}(t-M a_n),F_{0+}(t+M a_n)),B_{nrC}\bigr)<p_{j,M},
\end{equation}
where $p_{j,M}=d_1\exp(-d_2 M^3 \log j)$ for some $d_1,d_2>0$. Then 
$$
\Pr\bigl(\exists t\!\in\! (t_0,t_0+r)\!:\!\widehat F_{n+}(t)\!\notin\! [F_{0+}(t-M a_n),F_{0+}(t+M a_n)),B_{nrC}\bigr)\!\leq\! 
\sum\nolimits_{j=0}^{\infty} p_{j,M},
$$
and $\sum_{j=0}^{\infty} p_{j,M}<\epsilon/2$ under sufficiently large $M$. First, we consider 
\begin{equation*}
\Pr\bigl(\exists t\in I_{n,j}:\widehat F_{n+}(t)\geq F_{0+}(t+M a_n),B_{nr C}\bigr)\leq  \Pr(A_{njM}^+,B_{nrC}),
\end{equation*}
where $A_{njM}^+=\{\widehat F_{n+}(t_{n,j+1})\geq F_{0+}(s_{njM})\}$ and $s_{njM}=t_{n,j}+Ma_n$. Let $\tau_{nkj}^-$ be the last jump point of $F_{nk}$ before $t_{n,j+1}$, $k=1,\ldots,K$. On the event $B_{nrC}$ these jump points exists and $\tau_{nkj}^-\in [\tau_{nk1},t_{n,j+1})$. Hence, in notations of Lemma \ref{lem:martingale}, $\Pr(A_{njM}^+,B_{nrC})$ can be rewritten as follows:
\begin{equation}
\label{equ:pabpls}
\begin{array}{r}
\di 
\Pr\biggl(\bigcap_{k=1}^K\!\biggl\{\int_{[\tau_{nkj}^-,s_{njM})} \!\!\Bigl((\widehat F_{nk}(t)\!-\! F_{0k}(t))\!\! +\!\! \frac{F_{0k}(t)(\widehat F_{n+}\!(t)\!-\!F_{0+}(t))}{F_{0,K+1}(t)}\Bigr) dG(t)\vspace{+.5em}\\
\di
\leq \int_{[\tau_{nkj}^-,s_{njM})}  dM_{nk}(w)+{\cal R}_{nk}(\tau_{nk},s)\Bigr\},A_{njM}^+,B_{nrC}\biggr).
\end{array}
\end{equation}
Note that 
\begin{eqnarray*}
&&\int_{[t,s)} \frac{F_{0k}(u)(\widehat F_{n+}(u)\!-\!F_{0+}(u))}{F_{0,K+1}(u)}dG(u)\\
&&=
\frac{F_{0k}(s)}{F_{0,K+1}(s)}\int_{[t,s)} \bigl(\widehat F_{n+}(u)\!-\!F_{0+}(u)\bigr)dG(u)+\rho(t,s),
\end{eqnarray*}
where 
$
\rho(t,s)=\int_{[t,s)} \Bigl(\frac{F_{0k}(s)}{F_{0,K+1}(s)} -\frac{F_{0k}(u)}{F_{0,K+1}(u)}\Bigr) (\widehat F_{n+}(u)\!-\!F_{0+}(u))dG(u)
$.
Using Cauchy-Schwarz inequality and (\ref{equ:qrate}) we obtain that $|\rho(t,s)|$ is bounded above by
\begin{equation}
\label{equ:approrate}
\begin{array}{r}
\di
{\|\widehat F_{n+}-\widehat F_{0+}\|}_2\Bigl(\int_{[t,s)} \!\!\Bigl(\frac{F_{0k}(s)}{F_{0,K+1}(s)} -\frac{F_{0k}(u)}{F_{0,K+1}(u)}\Bigr)^2 dG(u)\Bigr)^{1/2} \\
\di
=O_P(n^{-1/3}(s-t)^{3/2}).
\end{array}
\end{equation}
Then (\ref{equ:pabpls}) can be rewritten as follows:
\begin{equation}
\label{equ:pab}
\begin{array}{l}
\di 
\Pr\biggl(\bigcap_{k=1}^K\biggl\{\int_{[\tau_{nkj}^-,s_{njM})} \!\!\Bigl((\widehat F_{nk}(t)\!-\! F_{0k}(t)) \\ 
\hspace{20.5ex}\di + \frac{F_{0k}(s_{njM})(\widehat F_{n+}(t)\!-\!F_{0+}(t))}{F_{0,K+1}(s_{njM})}\Bigr)dG(t)\\
\di\leq \int_{[\tau_{nkj}^-,s_{njM})}  dM_{nk}(w)+{\cal R}_{nk}^*(\tau_{nkj}^-,s_{njM})\Bigr\},A_{njM}^+,B_{nrC}\biggr),
\end{array}
\end{equation}
where $\sup\nolimits_{t,s\in V_{r}(t_0): t<s} (|{\cal R}_{nk}^*(t,s)|)= O_P(n^{-2/3}\vee n^{-1/3}(s-t)^{3/2})$. 
The event $A_{njM}^+$ implies 
$
\bigcup_{k=1}^n\{\widehat F_{nk}(t_{n,j+1})\geq F_{0k}(s_{njM})\}
$,
and for each $k\in\{1,\ldots,K\}$,
\begin{eqnarray*}
&&\{\widehat F_{nk}(t_{n,j+1})\geq F_{0k}(s_{njM})\}\\
&&=\Bigl\{\widehat F_{nk}(t_{n,j+1})\geq  F_{0k}(s_{njM}),\int_{\tau_{nkj}^-}^{s_{njM}}(\widehat F_{n+}(u)-F_{0+}(u))dG(u)\geq 0\Bigr\}\\
&& \hspace{+3.7ex}{\textstyle\bigcup}\;
\Bigl\{\widehat F_{nk}(t_{n,j+1})\geq F_{0k}(s_{njM}),\int_{\tau_{nkj}^-}^{s_{njM}}(\widehat F_{n+}(u)-F_{0+}(u))dG(u)<0\Bigr\}
\end{eqnarray*}
On \ the \ event \ $\{\widehat F_{nk}(t_{n,j+1})\geq F_{0k}(s_{njM})\}$ \ applying \ $\widehat F_{nk}(t)-F_{0k}(t)\geq F_{0k}(s_{njM})- F_{0k}(t)$ and bounded away from zero property for the derivatives $F_{0k}'(t)\vee G'(t) \geq 1/\varepsilon$ \ for \  all $t\in [\tau_{nkj}^-,s_{njM})$, \ we \ obtain \ that $\int_{[\tau_{nkj}^-,s_{njM})}(\widehat F_{nk}(t)-F_{0k}(t)) dG(t)\geq b(s_{njM}-\tau_{njk}^-)^2$ for some $b>0$ and all $k=1,\ldots,K$. Hence, the probability in (\ref{equ:pab}) is bounded above by 
\begin{equation}
\label{equ:mrtf1}
\begin{array}{rcl}
\di
\sum\nolimits_{k=1}^K \Pr\biggl(b(s_{njM}-\di\tau_{nkj}^-)^2&\leq&\di\int_{[\tau_{nkj}^-,s_{njM})}  dM_{nk}(w)\\
 &+&\di {\cal R}_{nk}^*(\tau_{nkj}^-,s_{njM}),A_{njM}^+,B_{nrC}\biggr)
\end{array}
\vspace{-1.5em}
\end{equation}
\begin{equation}
\label{equ:mrtf2}
+\; \Pr\Bigl(\bigcup_{k=1}^K\int_{\tau_{nkj}^-}^{s_{njM}}(\widehat F_{n+}(u)-F_{0+}(u))dG(u)<0,A_{njM}^+,B_{nrC}\Bigr).
\end{equation}
On the event in (\ref{equ:mrtf2}) there exists an $l\in 1,\ldots,K$ such that $\widehat F_{nl}(t_{n,j+1})\!\leq\! F_{0l}(s_{njM})$ and $\widehat F_{nk}(t_{n,j+1})\!>\! F_{0k}(s_{njM})$ for all $k$: $\tau_{nkj}^{-}>\tau_{nlj}^{-}$. Then the probability in (\ref{equ:mrtf2}) is bounded above by 
\begin{equation}
\label{equ:mrtf3}
\Pr\Bigl(\int_{\tau_{nlj}^-}^{s_{njM}}(\widehat F_{n+}(u)-F_{0+}(u))dG(u)<0,A_{njM}^+,B_{nrC}\Bigr)
\end{equation}
Applying Lemma \ref{lem:martbase} for each $k=1,\ldots,K$ we obtain from (\ref{equ:marti1}) that there exists positive constants $d_{11}$ and $d_{12}$, such that the sum in (\ref{equ:mrtf1}) is bounded above by $d_{11}\exp(-d_{12}M^3\log n)$. Slightly remaking \citet[proof of Lemma 4.14]{GMW08a} we get the upper bound $d_{21}\exp(-d_{22}M^3\log n)$ for the probability in (\ref{equ:mrtf3}) with some positives $d_{21}$ and $d_{22}$ (full proof of the bound is given in Appendix, Lemma A.\ref{lem:neg}). Hence, there exists $d_{1}^+$ and $d_2^+$ such that $\Pr(A_{njM}^+,B_{nrC})\leq d_{1}^+\exp(-d_2^+ M^3\log n)$. 

The \ required \ bound \ $\Pr(A_{njM}^-,B_{nrC})\leq d_{1}^-\exp(-d_2^- M^3\log n)$, \ where $A_{njM}^-=\{\widehat F_{n+}(t_{n,j+1})\leq F_{0+}(s_{njM}^*)\}$ and $s_{njM}^*=t_{n,j+1}-Ma_n$ can be obtained similarly. Then we get (\ref{equ:prt}), which implies together with $\Pr(B_{nrC})>1-\epsilon/2$ the inequality (\ref{equ:scl1}). 

In order to prove the lemma for $t_0$: $F_{0+}(t_0)>0$ we can split the problem and get the required uniform convergence rate separately over the right hand side $[t_0,t_0+r)$ and over the left hand side $(t_0-r,t_0)$ of the neighborhood $V_r(t_0)$ for some $r>0$. Moreover, the event $B_{nrC}$ should be changed to its two-sided form
\begin{eqnarray*}
B_{nrC}\!=\!\bigl\{\mbox{each of}\; F_{nk},k\!=\! 1,\ldots,K,\; \mbox{has at least one jump in}\; (t_0\!+\! r,t_0 \! + \! 2r)
\rule{+.5ex}{0em}&&\\
\mbox{and at least one jump in} \;  (t_0-2r,t_0-r),t_0+2r <T_{(n)},\rule{+.5ex}{0em}&& \\
\sup_{t_0-2r\leq t<s<t_0+2r} (\max_{k=1,\ldots,K}|{\cal R}_{nk}(t,s)|)\leq C n^{-2/3}\bigr\}.&&
\end{eqnarray*}
The rest of the proof for the right hand side is analogous to the case of $F_{0+}(t_0)=0$, and the proof for the left hand side   is symmetric to the right hand side one. In the last case the key to the proof is the inequality (\ref{equ:marti2}).

Alternatively, one can follow \citet[proof of Theorem 4.10 with $v_n(t)=n^{-1/3}\log^{1/3} n$]{GMW08a}. The lemma is proved.
\end{proof}

\begin{lemma}
\label{lem:crlunirate}
Under the conditions of Lemma \ref{lem:lunirate} there exists a constant $r>0$, such that for any $k\in\{1,\ldots,K\}$
\begin{equation*}
\sup\nolimits_{t\in V_{r}(t_0)} |\widehat F_{nk}(t)-F_{0k}(t)|=O_P(n^{-1/3}\log^{1/3} n).
\end{equation*}
\end{lemma}

\begin{proof}
By Lemma \ref{lem:lunirate}, for any point $t_0\in [0,G(\gamma)]$ there exists a constant $r>0$, such that 
$
\sup\nolimits_{t\in V_{r}(t_0)} |\widehat F_{n+}(t)-F_{0+}(t)|=O_P(n^{-1/3}\log^{1/3} n).
$ 
Taking account of $[0,G(\gamma)]$ is a compact set we obtain that 
$$
\sup_{t\in [0,\gamma]}|\widehat F_{n+}(t)\!-\!F_{0+}(t)|=\!\!\sup_{t\in [0,G(\gamma)]} |\widehat F_{n+}(t)\!-\! F_{0+}(t)|\!=\!O_P(n^{-1/3}\log^{1/3} n). 
$$
Then for any $\epsilon>0$ there exist an $n_0$ and $C^*>0$, such that $\Pr(D_{nC^*})>1-\epsilon/4$ for all $n\geq n_0$, where $D_{nC^*}=\bigl\{\|\widehat F_{n+}-F_{0+}\|_{[0,\gamma)}\leq C^* n^{-1/3}\log^{1/3} n\bigr\}$. Moreover, for any $\epsilon>0$ and $C>0$ there exists an $r>0$, such that $\Pr(B_{nrC})>1-\epsilon/4$ under sufficiently large $n$. Therefore, in order to prove the lemma it is sufficient to show that under $\Pr(D_{nC^*})\wedge\Pr(B_{nrC})>1-\epsilon/4$ for each $k\in\{1,\ldots,K\}$,
$$
\Pr\bigl(\exists t\!\in\! V_r(t_0)\!:\!\widehat F_{nk}\!\notin\! [F_{0k}(t-\! M a_n),F_{0k}(t+\! M a_n)),B_{nrC},D_{nC^*}\bigr)\!\leq\! 
\sum_{j=0}^{\infty} p_{j,M},
$$
and $\sum\nolimits_{j=0}^{\infty} p_{j,M}\to 0$ as $M\to\infty$, where $V_r(0)=(0,r)$, and $V_r(t)=(t-r,t+r)$ if $t>0$.   

Let $A_{njkM}^+=\{\widehat F_{n+}(t_{n,j+1})\geq F_{0+}(s_{njM})\}$, where $s_{njM}=t_{n,j}+Ma_n$ and $t_{n,j}=t_0+j n^{-1/3}$ as in Lemma \ref{lem:lunirate}. Analogously as in (\ref{equ:pab}) we can write $\Pr(A_{njkM}^+,B_{nrC},D_{nC^*})$ as follows: 
\begin{equation}
\label{equ:pab2}
\begin{array}{l}
\di 
\Pr\biggl(\biggl\{\int_{[\tau_{nkj}^-,s_{njM})}\! \Bigl((\widehat F_{nk}(t)-F_{0k}(t)) \\
\hspace{+24.3ex}
\di + \:\frac{F_{0k}(s_{njM})(\widehat F_{n+}(t)-F_{0+}(t))}{F_{0,K+1}(s_{njM})}\Bigr)dG(t)\\
\hspace{+3ex}
\di\leq \int_{[\tau_{nkj}^-,s_{njM})} \!\!\!\! dM_{nk}(w)\!+\! {\cal R}_{nk}^*(\tau_{nkj}^-,s_{njM})\Bigr\},A_{njkM}^+,B_{nrC},D_{nC^*}\biggr),
\end{array}
\end{equation}
where $\sup\nolimits_{t,s\in V_{r}(t_0): t<s} (|{\cal R}_{nk}^*(t,s)|)= O_P(n^{-2/3}\vee n^{-1/3}(s-t)^{3/2})$ and $\tau_{nkj}^-$ is the last jump point of $\widehat F_{nk}$ before $t_{n,j+1}$. 

On the event $A_{njkM}^+$ using similar arguments as in Lemma \ref{lem:lunirate} we obtain that $\int_{[\tau_{nkj}^-,s_{njM})}(\widehat F_{nk}(t)-F_{0k}(t)) dG(t)\geq b^*(s_{njM}-\tau_{njk}^-)^2$ for some $b^*>0$. On the other hand, 
$$
\int_{[\tau_{nkj}^-,s_{njM})}\!\!\!\!\!\! \frac{F_{0k}(s_{njM})(\widehat F_{n+}(t)\!-\! F_{0+}(t))}{F_{0,K+1}(s_{njM})}dG(t)\!\leq\! c\, \|\widehat F_{n+}\!\!-\! F_{0+}\|_{[0,\gamma)} (s_{njM}\!-\!\tau_{nkj}^-)
$$
under some $c>0$. Taking account of $(s_{njM}-\tau_{njk}^-)\geq (s_{njM}-t_{n,j+1})\geq (M-1)a_n$ we conclude that on  the event $D_{nC^*}$ for any fixed $C^*>0$ and sufficiently large $M$ the left hand side of the inequality under the probability sign in (\ref{equ:pab2}) is bounded below by $b(s_{njM}-\tau_{njk}^-)^2$ with some $b\in (0,b^*)$. Hence, 
$\Pr(A_{njkM}^+,B_{nrC},D_{nC^*})\leq  p_{jbM}=d_1\exp(-d_{2b} M^3\log n)$ for some $d_1,d_{2b}>0$ follows immediately from  Lemma \ref{lem:martbase}. Consequently, 
$$
\Pr\bigl(\exists t\!\in\! V_r(t_0)\!:\!\widehat F_{nk}> F_{0k}(t+M a_n),B_{nrC},D_{nC^*}\bigr)\!\leq\! 
\sum\nolimits_{j=0}^{\infty} p_{jbM}\to 0
$$
as $M\to\infty$. Under the case of $t_0>0$ the upper bound  
$$
\Pr\bigl(\exists t\!\in\! V_r(t_0)\!:\!\widehat F_{nk} < F_{0k}(t-M a_n),B_{nrC},D_{nC^*}\bigr)\!\leq\! 
\sum\nolimits_{j=0}^{\infty} p_{jM}\to 0,
$$
as $M\to\infty$ can be obtained analogously. The lemma is proved. 
\end{proof}

We continue with the proof of Theorem \ref{teo:unirate}. 

\begin{proof}[Proof of Theorem \ref{teo:unirate}] First we use the Smirnov's transformation to the observation time $T$, and consider the sample $(X_i^{(u)},Y_i^{(u)},T_i^{(u)})$, where $X_i^{(u)}=G(X_{i})$, $Y_i^{(u)}=G(Y_{i})$ and $T_i^{(u)}=G(T_i)$, $i=1,\ldots,n$. The log likelihood function (\ref{equ:llik0}) of the new sample $(T^{(u)}\!\!,\Delta^{(u)})$ of current status data with competing risks is 
$$
LL^{(u)}(F_{n}^{(u)})=LL\bigl(F_{n}^{(u)} | T^{(u)}\!\!,X^{(u)},Y^{(u)})=LL(F_{n} | T,X,Y),
$$
where $F_{n}^{(u)}=(F_{n1}^{(u)},\ldots,F_{nK}^{(u)})$, $F_{nk}^{(u)}\equiv F_{nk}\circ G^-$, $k=1,\ldots,K$, and $G^-$ is the generalized inverse function for $G$. 
Moreover, $F_{nk}^{(u)}$  are continuously differentiable on the interval $(0,G(\gamma)]$ with bounded and bounded away from zero derivatives. 

Then by Lemma \ref{lem:crlunirate}, for any point $t_0\in [0,G(\gamma)]$ there exists a constant $r>0$, such that 
$$
\sup\nolimits_{t\in V_{r}(t_0)} |\widehat F_{nk}^{(u)}(t)-F_{0k}^{(u)}(t)|=O_P(n^{-1/3}\log^{1/3} n),
$$ 
$k=1,\ldots,K$. Taking account of $[0,G(\gamma)]$ is a compact set we obtain that 
$$
\sup_{t\in [0,\gamma]}|\widehat F_{nk}^{(u)}(t)-F_{0k}^{(u)}(t)|=\!\!\!\sup_{t\in [0,G(\gamma)]} |\widehat F_{nk}^{(u)}(t)-F_{0k}^{(u)}(t)|=O_P(n^{-1/3}\log^{1/3} n). 
$$
The theorem is proved. 
\end{proof}

\begin{proof}[Proof of Corollary \ref{cor:unirate}]
We lose no generality in assuming that $G$ is the standard uniform distribution $U(0,1)$. Otherwise, we apply Smirnov's transformation as in the proof of Theorem \ref{teo:unirate}. In order to prove the corollary it will be sufficient to derive that (\ref{equ:unilocrate}) holds under $V_{r}=(1-r,1)$ with some $r>0$. 

Under the right-censored data ($F_{n}\equiv F_{n+}$) both the inequalities (\ref{equ:char1}) and (\ref{equ:char2}) with the equality holds if $\tau_{n,j}$ is a jump point of $\widehat F_{n}$ remain correct for all $s>0$. Note that the last jump point of the MLE is not uniquely defined if $\Delta_{(n)}=0$. Let $\tau_{n,\max}=\max\{T_i:\Delta_i=0\}$, for which $\widehat F_n(\tau_{n,\max})<1$. 
Applying \citet{yng77} we obtain that 
$$
\begin{array}{rcl}
\di
\Pr(\bar m\!<\! n\!-\! m)&\!=\!&\di\frac{n!}{(n-m)!}\\
&\!\times\! &\di \int_{\Rb^m}\Bigl\{\prod_{i=1}^{m} F_0(t_i)\Bigr\} G(t_m)^{n-m}\indi_{\{t_m\leq\ldots\leq t_1\}}dG(t_1)\cdots dG(t_m),

\end{array}
$$
where \ \ $\bar m=\max\{i:\Delta_{(i)}=0\}$. \ \
Hence, \ \ $\Pr(\Delta_{(n)}=0)\to 0$ \ \ and \ \ $\Pr(T_{(m_k)}>1-n^{-1/3})\to 1$ as $n\to\infty$.   

Similarly as in Lemma \ref{lem:martingale} we obtain that for any jump point $\tau_{n,i}$ of $\widehat F_n$ and all $s>0$,
$$
\int_{[\tau_{n,i},s)} \frac{\widehat F_n(t)-F_0(w)}{F_0(t)}dG(t)\leq \int_{[\tau_{n,i},s)}  dM_n(t)+{\cal R}_n(\tau_{n,i},s)
$$
with 
$
\sup_{1-2r\leq t<s\leq 1} (|{\cal R}_n(t,s)|)= O_P(n^{-2/3})
$
for some $r>0$, where 
$$
M_n(t)\!\!=\!\!\int_{u\leq t}\!\!\! \frac{(1\!-\! F_0(u))(\delta\!-\! F_0(u))}{F_0(u)}dP_n(u,\delta)\!-\!
\!\int_{u\leq t} \!\!\!(\bar\delta\!-\!(1\!-\! F_0(u)))dP_n(u,\delta).
$$
Now using arguments similar to the proof of Lemma \ref{lem:lunirate} (left hand side case) we conclude that (\ref{equ:unilocrate}) holds under $V_{r}=(1-r,1)$ with some $r>0$. Applying  Lemma \ref{lem:lunirate} for other points in $[0,1)$ we obtain the rate of uniform convergence $O_P(n^{-1/3}\log^{1/3} n)$ on the interval $[0,1]$. The corollary is proved. 
\end{proof} 

\begin{proof}[Proof of Corollary \ref{teo:funirate}] \ \  By the reconstruction formula (\ref{equ:reconstr}) and the Duhamel equation (see e.g. \citet{abgk93}),
\begin{equation*}
\widehat S_n(x)-S(x)=S(x) \int_0^x \frac{\widehat S_n(u_-)}{S(u)} \Bigl(\frac{d\widehat F_{n1}(u)}{1-\widehat F_{n+}(u_-)}-\frac{dF_{01}(u)}{1-F_{0+}(u)}\Bigr). 
\end{equation*}
Hence,
\begin{eqnarray*}
|\widehat S_n(x)\!-\! S(x)|\! &\leq \! & \Bigl|\int_0^x \frac{\widehat S_n(u_-)}{1-\widehat F_{n+}(u_-)} d(\widehat F_{n1}(u)-F_{01}(u)) \Bigr| \\
&+\! &\int_0^x \frac{\widehat S_n(u_-) |\widehat F_{n+}(u_-)\!-\! F_{0+}(u)| dF_{01}(u)}{(1-\widehat F_{n+}(u_-))(1-\widehat F_{n+}(u))} \!=\! I_{1n}(x)\! +\! I_{2n}(x).
\end{eqnarray*}
Note that $\widehat S_n(u_-)/(1-\widehat F_{n+}(u_-))=1/\widehat Q_n(u_-)$ is a non decreasing function,  and $d\widehat F_{n+}=-\widehat Q_{n-}d\widehat S_n-\widehat S_{n}d\widehat Q_n$. Using the integration by parts formula we have
\begin{eqnarray*}
I_{1n}(x)&\leq & \frac{|\widehat F_{n1}(x)-F_{01}(x)|}{\widehat Q_n(\gamma)} -\int_{0}^{\gamma} \frac{|\widehat F_{n1}(u)-F_{01}(u)| }{\widehat Q_n(u)\widehat Q_n(u_-)}d\widehat Q_n(u) \\
&\leq & M_n |\widehat F_{n1}(x)-F_{01}(x)| + M_n^3 \int_0^{\gamma} |\widehat F_{n1}-F_{01}| d\widehat F_{n+},
\end{eqnarray*}
and 
$$
I_{2n}(x)\leq M_n^2 \int_{0}^{\gamma} |\widehat F_{n+}(u_-)-F_{0+}(u)| dF_{01}(u)=M_n^2 \int_{0}^{\gamma} |\widehat F_{n+}-F_{0+}| dF_{01} 
$$
for a positive constant $M_n\geq (1-\widehat F_{n+}(\gamma))^{-1}$. 
Applying Theorem \ref{teo:unirate} together with consistency of the estimator $\widehat F_{n+}$ and $F_{0+}(\gamma)<1$ we obtain  (\ref{equ:rate}). The corollary is proved. 
\end{proof}

\appendix

\section{A technical proof}
\label{app}

\begin{lemmaa}
\label{lem:neg}
Under the conditions of Lemma \ref{lem:lunirate} there exists $d_1,d_2>0$, such that 
\begin{equation}
\label{equ:sbound}
\Pr\Bigl(\int_{\tau_{nlj}^-}^{s_{njM}}(\widehat F_{n+}(u)-F_{0+}(u))dG(u)<0,A_{njM}^+,B_{nrC}\Bigr)\leq p_{jM},
\end{equation}
where \ $p_{jM}\!=\!d_1\exp(-d_2 M^3\log^{1/3} n)$ \ and \ $l\in 1,\ldots,K$ \ is \ such \ that $\widehat F_{nl}(t_{n,j+1})\!\leq\! F_{0l}(s_{njM})$ and $\widehat F_{nk}(t_{n,j+1})\!>\! F_{0k}(s_{njM})$ for all $k$: $\tau_{nkj}^{-}>\tau_{nlj}^{-}$.
\end{lemmaa}

\begin{proof}
We lose no generality by the assumption $\tau_{n1j}^-\leq\ldots\leq\tau_{nKj}^-$. On the event $A_{njM}^+$ let $l^*=l$ if $\int_{\tau_{nlj}^-}^{\tau_{nkj}^-} (\widehat F_{n+}-F_{0+})dG\leq 0$ for all $k>l$, and $l^*=\max\bigl\{k\in l+1,\ldots,K :\int_{\tau_{nlj}^-}^{\tau_{nkj}^-}(\widehat F_{n+}-F_{0+})dG>0\bigr\}$. For any fixed $l\leq l^*$ using notation $\tau_{n,K+1,j}^-=s_{njM}$ we can write that 
\begin{eqnarray*}
\int_{\tau_{nl^*j}^-}^{s_{njM}} (\widehat F_{n+}-F_{0+})dG&=&\sum\nolimits_{k=l^*+1}^K \int_{\tau_{nl^*j}^-}^{\tau_{nkj}^-} (\widehat F_{nk}-F_{0k})dG\\
&+&
 \sum\nolimits_{k=l^*}^K \sum\nolimits_{p=1}^k \int_{\tau_{nkj}^-}^{\tau_{n,k+1,j}^-} (\widehat F_{np}-F_{0p})dG
\end{eqnarray*}
Using (\ref{equ:mart-b})  and (\ref{equ:approrate}) we obtain that for each $k=1,\ldots,K$ and $t>T_{(1)}$,
\begin{eqnarray*}
\int_{t}^{\tau_{nkj}^-}(\widehat F_{nk}-F_{0k})dG &+&\frac{F_{0k}(s_{njM})}{F_{0,K+1}(s_{njM})}\int_{t}^{\tau_{nkj}^-}(\widehat F_{n+}-F_{0+})dG \\
&\geq & \int_{t}^{\tau_{nkj}^-} dM_{nk}- C(n^{-2/3}\vee n^{-1/3}(\tau_{nkj}^- -t)^{3/2}).
\end{eqnarray*}
Using notations of $l$ and $l^*$, and $\int_{\tau_{nl^*j}^-}^{\tau_{nkj}^-}=\int_{\tau_{nlj}^-}^{\tau_{nkj}^-}-\int_{\tau_{nlj}^-}^{\tau_{nl^*j}^-}$ we conclude that 
$
\int_{\tau_{nl^*j}^-}^{\tau_{nkj}^-}(\widehat{F}_{n+}-F_{0+})dG\leq 0
$
for all $k=l^*+1,\ldots,K$. Then under fixed $l,l^*:l\leq l^*$ on the events $A_{njM}^+$ and $B_{nrC}$,
\begin{equation}
\label{equ:msum}
\begin{array}{l}
\displaystyle
\sum\nolimits_{k=l^*+1}^K \int_{\tau_{nl^*j}^-}^{\tau_{nkj}^-}(\widehat F_{nk}-F_{0k})dG\\ 
\displaystyle
\geq \sum\nolimits_{k=l^*+1}^K \int_{[\tau_{nl^*j}^-,\tau_{nkj}^-)} dM_{nk}- C(n^{-2/3}\vee n^{-1/3}(\tau_{nkj}^- -\tau_{nl^*j}^-)^{3/2}).
\end{array}
\end{equation}
By definition of $l$, 
$
\sum_{p=k+1}^K \widehat F_{np}(t_{n,j+1}) >\sum_{p=k+1}^K \widehat F_{np}(s_{njM}) 
$
for all $k=l,\ldots,K$. Then on the event $A_{njM}^+$, 
$
\sum_{p=1}^K \widehat F_{np}(t_{n,j+1}) >\sum_{p=k+1}^K \widehat F_{np}(s_{njM}) 
$.
Moreover, taking account of $\tau_{n1j}^-\leq\ldots\leq\tau_{nKj}^-$  we have that for all $k=l,\ldots,K$ and $u\geq \tau_{nkj}^-$,
$$
\sum\nolimits_{p=1}^k F_{np}(u)\geq \sum\nolimits_{p=1}^k F_{np}(\tau_{nkj}^-)\geq \sum\nolimits_{p=1}^k F_{np}(\tau_{npj}^-) > \sum\nolimits_{p=1}^k F_{0p}(s_{njM})
$$
Hence,
\begin{equation}
\label{equ:msum2}
\begin{array}{l}
\displaystyle
\sum\nolimits_{k=l^*}^K \sum_{p=1}^k \int_{\tau_{nkj}^-}^{\tau_{n,k+1,j}^-} \!\!(\widehat F_{np}-F_{0p})dG\\
\displaystyle
\hspace{+20ex}\geq 
\sum\nolimits_{k=1}^K \int_{\tau_{nkj}^-\vee \tau_{nl^*j}^-}^{s_{njM}}\!\!\!(F_{0k}(s_{njM})-F_{0k}(u))dG(u). 
\end{array}
\end{equation}
Let $D_{l_0l^*_0\tau}=\{l=l_0,l^*=l^*_0,\tau_{n1j}^-\leq\ldots\leq\tau_{nKj}^-\}$. Using (\ref{equ:msum}) and (\ref{equ:msum2}) we obtain that
\begin{eqnarray*}
&&\Pr\Bigl(\int_{\tau_{nlj}^-}^{s_{njM}}(\widehat F_{n+}(u)-F_{0+}(u))dG(u)<0,A_{njM}^+,B_{nrC},l=l_0,D_{ll^*\tau}\Bigr)\\
&&
\leq 
\Pr\biggl( \sum_{k=l^*+1}^K \int_{[\tau_{nkj}^-,s_{njM})}  dM_{nk}-\sum_{k=l^*+1}^K \int_{[\tau_{nl^*j}^-,s_{njM})} dM_{nk} \\
&& \hspace{8ex} - C^*(n^{-2/3}\vee n^{-1/3}(s_{njM}-\tau_{nl^*j}^-)^{3/2})
\\
&& \hspace{8ex}
+\sum\nolimits_{k=1}^K \int_{\tau_{nkj}^-\vee \tau_{nl^*j}^-}^{s_{njM}}\!\!\!\!(F_{0k}(s_{njM})-F_{0k}(u))dG(u)\leq 0,B_{nrC},D_{ll^*\tau}\biggr).
\end{eqnarray*}
It is clear, $\int_{\tau_{nkj}^-\vee \tau_{nl^*j}^-}^{s_{njM}}\!\!\!\!(F_{0k}(s_{njM})-F_{0k}(u))dG(u)\geq b(s_{njM}-\tau_{nl^*j}^-)^2\geq b(M-1)n^{-1/3}\log^{1/3} n$ for some $b>0$. Then for any $b_-<b$ under the sufficiently large $n$,
\begin{eqnarray*}
&&\Pr\Bigl(\int_{\tau_{nlj}^-}^{s_{njM}}(\widehat F_{n+}(u)-F_{0+}(u))dG(u)<0,A_{njM}^+,B_{nrC},D_{ll^*\tau}\Bigr)\\
&&
\leq 
\Pr\biggl( \sum\nolimits_{k=l^*+1}^K \int_{[\tau_{nkj}^-,s_{njM})}  dM_{nk}-\sum\nolimits_{k=l^*+1}^K \int_{[\tau_{nl^*j}^{-},s_{njM})}dM_{nk}\\
&&\hspace{+32.8ex} +\: b_-(s_{njM}-\tau_{nl^*j}^-)\leq 0,B_{nrC},D_{ll^*\tau} 
\biggr).
\end{eqnarray*}
Note that for any fixed $l,l^*:l\leq l^*$ under $\tau_{n1j}^-\leq\ldots\leq\tau_{nKj}^-$ the right hand side of the last inequality is bounded above by 
$$ 
\Pr\biggl( \sum_{k=l^*+1}^K \int_{[\tau_{nkj}^-,s_{njM})} \!\!\!\!\!\!\!\!\! dM_{nk}-\!\!\!\sum_{k=l^*+1}^K \int_{[\tau_{nl^*j}^{-},s_{njM})} \!\!\!\!\!\! \!\! \! dM_{nk}+b_-(s_{njM}-\tau_{nl^*j}^-)\!\leq\! 0,B_{nrC}\!\biggr).
$$
The same bound for the probability left hand side of (\ref{equ:sbound}) holds under $\tau_{n\sigma_1j}^-\leq\ldots\leq\tau_{n\sigma_Kj}^-$ for any permutation $\sigma=(\sigma_1,\ldots,\sigma_K)$ of the indices $(1,\ldots,K)$. 
Finally, applying Lemma \ref{lem:martingale} for each $l,l^*$ and $\sigma$ several times and conbinig results by the total probability formula we get (\ref{equ:sbound}). The lemma is proved.
\end{proof}

\bibliography{ArXiv_base.bib}

\begin{thebibliography}{17}

\bibitem[\protect\citeauthoryear{{Andersen} et~al.}{1993}]{abgk93}
\begin{bbook}[author]
\bauthor{\bsnm{{Andersen}},~\bfnm{Per~Kragh}\binits{P.~K.}},
  \bauthor{\bsnm{{Borgan}},~\bfnm{{\O}rnulf}\binits{{\O}.}},
  \bauthor{\bsnm{{Gill}},~\bfnm{Richard~D.}\binits{R.~D.}} \AND
  \bauthor{\bsnm{{Keiding}},~\bfnm{Niels}\binits{N.}}
(\byear{1993}).
\btitle{Statistical Models Based on Counting Processes.}
\bpublisher{New York: Springer-Verlag}.
\end{bbook}
\endbibitem

\bibitem[\protect\citeauthoryear{{Ayer} et~al.}{1955}]{abers55}
\begin{barticle}[author]
\bauthor{\bsnm{{Ayer}},~\bfnm{Miriam}\binits{M.}},
  \bauthor{\bsnm{{Brunk}},~\bfnm{H.~D.}\binits{H.~D.}},
  \bauthor{\bsnm{{Ewing}},~\bfnm{G.~M.}\binits{G.~M.}},
  \bauthor{\bsnm{{Reid}},~\bfnm{W.~T.}\binits{W.~T.}} \AND
  \bauthor{\bsnm{{Silverman}},~\bfnm{Edward}\binits{E.}}
(\byear{1955}).
\btitle{An Empirical Distribution Function for Sampling with Incomplete
  Information}.
\bjournal{The Annals of Mathematical Statistics}
\bvolume{26}
\bpages{641--647}.
\end{barticle}
\endbibitem

\bibitem[\protect\citeauthoryear{{Groeneboom}}{1987}]{gre87}
\begin{btechreport}[author]
\bauthor{\bsnm{{Groeneboom}},~\bfnm{P.}\binits{P.}}
(\byear{1987}).
\btitle{Asymptotics for Interval Censored Observations}
\btype{Technical Report} No. \bnumber{87-18},
\binstitution{Department of Mathematics, University of Amsterdam}.
\end{btechreport}
\endbibitem

\bibitem[\protect\citeauthoryear{{Groeneboom}}{1991}]{gre91}
\begin{btechreport}[author]
\bauthor{\bsnm{{Groeneboom}},~\bfnm{P.}\binits{P.}}
(\byear{1991}).
\btitle{Nonparametric maximum likelihood estimators for interval censoring and
  deconvolution}
\btype{Technical Report} No. \bnumber{378},
\binstitution{Department of Statistics, Stanford University}.
\end{btechreport}
\endbibitem

\bibitem[\protect\citeauthoryear{{Groeneboom}, {Jongbloed} and
  {Wellner}}{2008}]{GJW08}
\begin{barticle}[author]
\bauthor{\bsnm{{Groeneboom}},~\bfnm{P.}\binits{P.}},
  \bauthor{\bsnm{{Jongbloed}},~\bfnm{G.}\binits{G.}} \AND
  \bauthor{\bsnm{{Wellner}},~\bfnm{J.~A.}\binits{J.~A.}}
(\byear{2008}).
\btitle{The support reduction algorithm for computing non-parametric function
  estimates in mixture models}.
\bjournal{Scandinavian Journal of Statistics}
\bvolume{35}
\bpages{385--399}.
\end{barticle}
\endbibitem

\bibitem[\protect\citeauthoryear{Groeneboom and Jongbloed}{2014}]{GJ14}
\begin{bbook}[author]
\bauthor{\bsnm{Groeneboom},~\bfnm{P.}\binits{P.}} \AND
  \bauthor{\bsnm{Jongbloed},~\bfnm{G.}\binits{G.}}
(\byear{2014}).
\btitle{Nonparametric Estimation under Shape Constraints: Estimators,
  Algorithms and Asymptotics}.
\bseries{Cambridge Series in Statistical and Probabilistic Mathematics}.
\bpublisher{Cambridge University Press}.
\end{bbook}
\endbibitem

\bibitem[\protect\citeauthoryear{{Groeneboom}, {Maathuis} and
  {Wellner}}{2008a}]{GMW08a}
\begin{barticle}[author]
\bauthor{\bsnm{{Groeneboom}},~\bfnm{P.}\binits{P.}},
  \bauthor{\bsnm{{Maathuis}},~\bfnm{M.~H.}\binits{M.~H.}} \AND
  \bauthor{\bsnm{{Wellner}},~\bfnm{J.~A.}\binits{J.~A.}}
(\byear{2008}a).
\btitle{Current status data with competing risks: Consistency and rates of
  convergence of the MLE}.
\bjournal{Ann. Statist.}
\bvolume{36}
\bpages{1031--1063}.
\end{barticle}
\endbibitem

\bibitem[\protect\citeauthoryear{{Groeneboom}, {Maathuis} and
  {Wellner}}{2008b}]{GMW08b}
\begin{barticle}[author]
\bauthor{\bsnm{{Groeneboom}},~\bfnm{P.}\binits{P.}},
  \bauthor{\bsnm{{Maathuis}},~\bfnm{M.~H.}\binits{M.~H.}} \AND
  \bauthor{\bsnm{{Wellner}},~\bfnm{J.~A.}\binits{J.~A.}}
(\byear{2008}b).
\btitle{Current status data with competing risks: Limiting distribution of the
  MLE}.
\bjournal{Ann. Statist.}
\bvolume{36}
\bpages{1064--1089}.
\end{barticle}
\endbibitem

\bibitem[\protect\citeauthoryear{{Groeneboom} and {Wellner}}{1992}]{GW92}
\begin{bbook}[author]
\bauthor{\bsnm{{Groeneboom}},~\bfnm{P.}\binits{P.}} \AND
  \bauthor{\bsnm{{Wellner}},~\bfnm{J.~A.}\binits{J.~A.}}
(\byear{1992}).
\btitle{Information Bounds and Non-parametric Maximum Likelihood Estimation}.
\bseries{DMV Seminar Band 19}.
\bpublisher{Birkhauser Verlag}, \baddress{Basel}.
\end{bbook}
\endbibitem

\bibitem[\protect\citeauthoryear{{Hudgens}, {Satten} and
  {Longini}}{2001}]{HSL01}
\begin{barticle}[author]
\bauthor{\bsnm{{Hudgens}},~\bfnm{M.~G.}\binits{M.~G.}},
  \bauthor{\bsnm{{Satten}},~\bfnm{G.~A.}\binits{G.~A.}} \AND
  \bauthor{\bsnm{{Longini}},~\bfnm{I.~M.}\binits{I.~M.}}
(\byear{2001}).
\btitle{Nonparametric Maximum Likelihood Estimation for Competing Risks
  Survival Data Subject to Interval Censoring and Truncation}.
\bjournal{Biometrics}
\bvolume{57}
\bpages{74-80}.
\end{barticle}
\endbibitem

\bibitem[\protect\citeauthoryear{{Jewel}l, van~der {Laan} and
  {Henneman}}{2003}]{JLH03}
\begin{barticle}[author]
\bauthor{\bsnm{{Jewel}l},~\bfnm{N.~P.}\binits{N.~P.}},
  \bauthor{\bparticle{van~der} \bsnm{{Laan}},~\bfnm{M.}\binits{M.}} \AND
  \bauthor{\bsnm{{Henneman}},~\bfnm{T.}\binits{T.}}
(\byear{2003}).
\btitle{Nonparametric Estimation from Current Status Data with Competing
  Risks}.
\bjournal{Biometrika}
\bvolume{90}
\bpages{183--197}.
\end{barticle}
\endbibitem

\bibitem[\protect\citeauthoryear{{Maathuis}}{2013}]{mth13}
\begin{bmanual}[author]
\bauthor{\bsnm{{Maathuis}},~\bfnm{M.}\binits{M.}}
(\byear{2013}).
\btitle{MLEcens: Computation of the MLE for bivariate (interval) censored data}
\bnote{R package version 0.1-4}.
\end{bmanual}
\endbibitem

\bibitem[\protect\citeauthoryear{{Malov}}{2019}]{mal19}
\begin{barticle}[author]
\bauthor{\bsnm{{Malov}},~\bfnm{S.~V.}\binits{S.~V.}}
(\byear{2019}).
\btitle{Nonparametric estimation for a current status right-censored data
  model}.
\bjournal{Statistica Neerlandica}
\bpages{1-21}.
\bdoi{10.1111/stan.12180}
\end{barticle}
\endbibitem

\bibitem[\protect\citeauthoryear{{Turnbull}}{1974}]{tur74}
\begin{barticle}[author]
\bauthor{\bsnm{{Turnbull}},~\bfnm{Bruce~W.}\binits{B.~W.}}
(\byear{1974}).
\btitle{Nonparametric Estimation of a Survivorship Function with Doubly
  Censored Data.}
\bjournal{{Journal of the American Statistical Association}}
\bvolume{69}
\bpages{169--173}.
\end{barticle}
\endbibitem

\bibitem[\protect\citeauthoryear{{Turnbull}}{1976}]{tur76}
\begin{barticle}[author]
\bauthor{\bsnm{{Turnbull}},~\bfnm{Bruce~W.}\binits{B.~W.}}
(\byear{1976}).
\btitle{The Empirical Distribution Function with Arbitrarily Grouped, Censored
  and Truncated Data.}
\bjournal{Journal of the Royal Statistical Society. Series B}
\bvolume{38}
\bpages{290--295}.
\end{barticle}
\endbibitem

\bibitem[\protect\citeauthoryear{{van de Geer}}{2000}]{vdg00}
\begin{bbook}[author]
\bauthor{\bsnm{{van de Geer}},~\bfnm{S.~A.}\binits{S.~A.}}
(\byear{2000}).
\btitle{Empirical Processes in M-Estimation}.
\bseries{Cambridge Series in Statistical and Probabilistic Mathematics}.
\bpublisher{Cambridge University Press}.
\end{bbook}
\endbibitem

\bibitem[\protect\citeauthoryear{{Yang}}{1977}]{yng77}
\begin{barticle}[author]
\bauthor{\bsnm{{Yang}},~\bfnm{S.~S.}\binits{S.~S.}}
(\byear{1977}).
\btitle{General Distribution Theory of the Concomitants of Order Statistics}.
\bjournal{Ann. Statist.}
\bvolume{5}
\bpages{996--1002}.
\end{barticle}
\endbibitem

\end{thebibliography}

\end{document}